\documentclass{article}
\usepackage{amsmath, amsfonts, amsthm, amssymb}
\usepackage{graphicx}
\usepackage{float}
\usepackage{verbatim}

\numberwithin{equation}{section}
\newtheorem{thm}{Theorem}[section]
\newtheorem{lma}[thm]{Lemma}
\newtheorem{cor}[thm]{Corollary}
\newtheorem{defn}[thm]{Definition}

\renewcommand{\geq}{\geqslant}
\renewcommand{\leq}{\leqslant}

\renewcommand{\P}{\text{P}}

\title{On the packing dimension of box-like self-affine sets in the plane}

\author{Jonathan M. Fraser\\ \\
\emph{Mathematical Institute, University of St. Andrews, North Haugh,}\\ \emph{St. Andrews, Fife, KY16 9SS, Scotland}\\ \emph{e-mail: jmf32@st-andrews.ac.uk}}

\begin{document}
\maketitle

\begin{abstract}
We consider a class of planar self-affine sets which we call ``box-like''.  A box-like self-affine set is the attractor of an iterated function system (IFS) of affine maps where the image of the unit square, $[0,1]^2$, under arbitrary compositions of the maps is a rectangle with sides parallel to the axes.  This class contains the Bedford-McMullen carpets and the generalisations thereof considered by Lalley-Gatzouras, Bara\'nski and Feng-Wang as well as many other sets.  In particular, we allow the mappings in the IFS to have non-trivial rotational and reflectional components.  Assuming a rectangular open set condition, we compute the packing and box-counting dimensions by means of a pressure type formula based on the singular values of the maps.
\\ \\
\emph{Mathematics Subject Classification} 2010:  primary: 28A80, secondary: 28A78, 15A18.
\\ \\
\emph{Key words and phrases}:  packing dimension, box dimension, self-affine, singular value function, subadditivity, projections.
\end{abstract}

\section{Introduction}

The \emph{singular values} of a linear map, $A:\mathbb{R}^n \to \mathbb{R}^n$, are the positive square roots of the eigenvalues of $A^T A$.  Geometrically these numbers represent the lengths of the semi-axes of the image of the unit ball under $A$.  Thus, roughly speaking, the singular values correspond to how much the affine map contracts (or expands) in different directions.  For $s \in [0,n]$ define the \emph{singular value function} $\phi^s(A)$ by
\begin{equation} \label{singvaluefunc}
\phi^s(A) = \alpha_1 \alpha_2 \dots \alpha_{\lceil s \rceil-1} \alpha_{\lceil s \rceil}^{s-\lceil s \rceil+1}
\end{equation}
where $\alpha_1 \geq  \dots \geq \alpha_n$ are the singular values of $A$.  This function has played a vital r\^ole in the study of self-affine sets over the past 25 years and in this paper we introduce a modified singular value function (\ref{modsing}) which is designed specifically to compute the packing dimension.  Given an iterated function system (IFS) consisting of contracting affine maps, $\{A_i+t_i\}_{i=1}^m$, where the $A_i$ are linear contractions and the $t_i$ are translation vectors, it is well-known that there exists a unique non-empty compact set $F$ satisfying
\[
F = \bigcup_{i=1}^m S_i(F)
\]
which is termed the \emph{self-affine} attractor of the IFS.  Let $\mathcal{I}^k$ denote the set of all sequences $(i_1, \dots, i_k)$, where each $i_j \in \{1, \dots, m\}$, and let
\begin{equation} \label{affinitydim}
d (A_1, \dots, A_m) = \inf \bigg\{ s: \sum_{k=1}^\infty \sum_{\mathcal{I}^k}  \phi^s ( A_{i_1} \circ \dots \circ A_{i_k} ) < \infty \bigg\}.
\end{equation}
This number is called the \emph{affinity dimension} of $F$ and is always an upper bound for the upper box dimension of $F$.  Moreover, Falconer proved the following result in the seminal paper \cite{affine}, published in 1988.  We write $\prod_{i=1}^{m} \mathcal{L}^n$ to denote the $m$-fold product of $n$-dimensional Lebesgue measure, supported on the space $\times_{i=1}^{m}\mathbb{R}^n$, and $\dim_\text{{B}}, \dim_\text{{P}}$ and $\dim_\text{{H}}$ to denote the box-counting, packing and Hausdorff dimensions, respectively.

\begin{thm}\label{falconersthm}
Let $A_1, \dots, A_m$ be contracting linear self-maps on $\mathbb{R}^n$ with Lipshitz constants strictly less than $1/2$.  Then, for $\big(\prod_{i=1}^{m} \mathcal{L}^n\big)$-almost all $(t_1, \dots, t_m) \in \times_{i=1}^{m}\mathbb{R}^n$, the unique non-empty compact set $F$ satisfying
\[
F = \bigcup_{i=1}^{m}\, (A_i+t_i)(F)
\]
has
\[
\dim_\text{\emph{B}} F = \dim_\text{\emph{P}} F = \dim_\text{\emph{H}} F = \min \big\{ n, d \, \big(A_1, \dots, A_m \big) \big\}.
\]
\end{thm}

In fact, the initial proof required that the Lipshitz constants be strictly less than $1/3$ but this was relaxed to $1/2$ by Solomyak \cite{solomyak} who also showed that $1/2$ is the optimal constant.
\\ \\
Despite the elegance of the above result, it seems difficult to calculate the exact dimension of a self-affine set in general.  However, exact calculation of dimension in certain `exceptional cases' has attracted a great deal of interest in recent years.  The first example was the Bedford-McMullen carpet.  Take the unit square, $[0,1]^2$, and divide it up into an $m\times n$ grid for some $m,n \in \mathbb{N}$ with $1 < m \leq n$.  Then select a subset of the rectangles formed by the grid and consider the IFS consisting of the affine maps which map $[0,1]^2$ onto each chosen rectangle, preserving orientation.  Bedford \cite{bedford} and McMullen \cite{mcmullen} independently obtained explicit formulae for the box-counting, packing and Hausdorff dimensions of the attractor.  In general the Hausdorff dimension and box dimension can be different and can be strictly less than the affinity dimension.  However, if the maps are chosen such that the projection onto the horizontal axis is an interval (having dimension 1), then the box dimension equals the affinity dimension.  Our results help to formalise this observation for a much larger class of self-affine sets, see Corollaries \ref{coro2} and \ref{coro3}.

\begin{figure}[H]
	\centering
	\includegraphics[width=100mm]{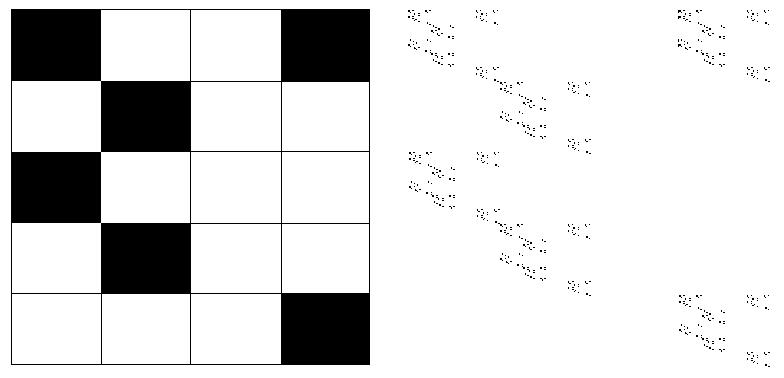}
\caption{A self-affine Bedford-McMullen carpet with $m=4$, $n=5$.  The shaded rectangles on the left indicate the 6 maps in the IFS.}
\end{figure}

Gatzouras and Lalley \cite{lalley-gatz} generalised the Bedford-McMullen construction by allowing the columns to have varying widths and be divided up, independently, with the only restriction being that the base of each rectangle had to be greater than or equal to the height.  Bara\'nski \cite{baranski} divided the unit square up into an arbitrary mesh of rectangles by slicing horizontally and vertically a finite number of times (at least once in each direction).  Also, Feng and Wang \cite{fengaffine} considered a construction where the rectangles did not have to be `aligned' as in the Bara\'nski type IFSs.  This added complication meant that the box dimension of the attractor was given in terms of the dimensions of its projection onto the horizontal and vertical axes.

\begin{figure}[H]
	\centering
	\includegraphics[width=150mm]{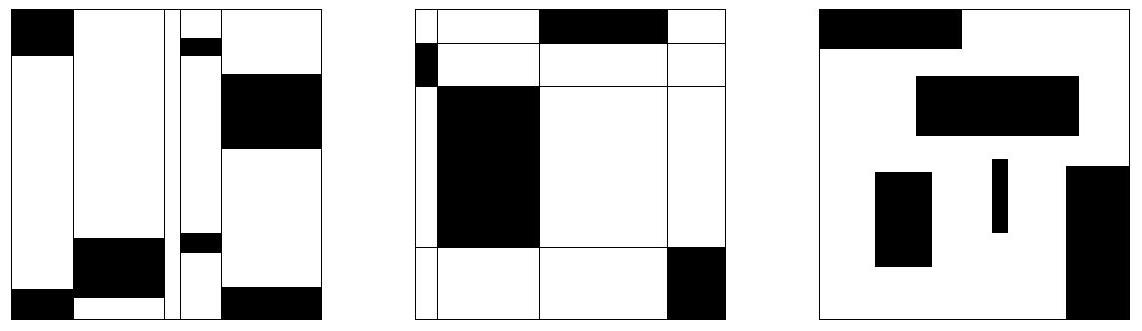}
\caption{Three examples of IFSs of the types considered by Gatzouras-Lalley, Bara\'nski and Feng-Wang, respectively.  The shaded rectangles represent the affine maps.}
\end{figure}

In all of the aforementioned examples the affine maps are orientation preserving.  In this paper we relax this requirement by allowing the maps to have non-trivial rotational and reflectional components.  We refer to the attractors of such systems as ``box-like'' sets and give their formal definition in Section \ref{boxdef}.  In Section \ref{results} we compute the packing and box-counting dimensions by means of a pressure type formula based on the singular values of the maps.  As in \cite{fengaffine} the dimension of projections will be significant.

\subsection{Box-like self-affine sets and notation} \label{boxdef}

We call a self-affine set \emph{box-like} if it is the attractor of an IFS of affine maps where the image of the unit square, $[0,1]^2$, under arbitrary compositions of the maps is a rectangle with sides parallel to the axes.  The affine maps which make up such an IFS are necessarily of the form $S = T\circ L + t$, where $T$ is a contracting linear map of the form
\[
T = \left ( \begin{array}{cc}
a & 0\\ 
0 &  b\\
\end{array} \right ) 
\]
for some $a,b \in (0,1)$; $L$ is a linear isometry of the plane for which $L([-1,1]^2) =[-1,1]^2$; and $t \in \mathbb{R}^2$ is a translation vector.  The following separation condition, which we will need to obtain the lower bound in our dimension result, was introduced in \cite{fengaffine}.

\begin{defn}
An IFS $\{S_i\}_{i=1}^m$ satisfies the rectangular open set condition (ROSC) if there exists a non-empty open rectangle, $R = (a,b)\times(c,d) \subset \mathbb{R}^2$, such that $\{S_i(R)\}_{i=1}^m$ are pairwise disjoint subsets of $R$.
\end{defn}
If, for all maps in the IFS, we let $L$ be the identity map and assume the \emph{rectangular open set condition}, then we obtain the class of self-affine sets considered by Feng and Wang \cite{fengaffine}.  
\\ \\
Let $\{S_i \}_{i \in \mathcal{I}}$ be an IFS consisting of maps of the form described above for some finite index set $\mathcal{I}$, with $\lvert \mathcal{I} \rvert \geq 2$, and let $F$ be the corresponding attractor, i.e., the unique non-empty compact set satisfying
\[
F = \bigcup_{i \in \mathcal{I}} S_{i}(F).
\]
We refer to $F$ as the \emph{box-like self-affine set}.  It is clear that we may choose a compact square $\mathcal{Q}\subset \mathbb{R}^2$ such that $\bigcup_{i \in \mathcal{I}} S_{i}(\mathcal{Q}) \subseteq \mathcal{Q}$.  Without loss of generality we will assume throughout that we may choose $\mathcal{Q} = [0,1]^2$.  Let
\[
\mathcal{I}_A = \{ i \in \mathcal{I} : S_i \text{ maps horizontal lines to horizontal lines} \}
\]
and
\[
\mathcal{I}_B = \{ i \in \mathcal{I} : S_i \text{ maps horizontal lines to vertical lines} \}.
\]
If $\mathcal{I}_B=\emptyset$, then we will say $F$ is of \emph{separated type} and otherwise we will say that $F$ is of \emph{non-separated type}.  It will become clear why we make this distinction in the following section.  
\\ \\
Write $\mathcal{I}^* = \bigcup_{k\geq1} \mathcal{I}^k$ to denote the set of all finite sequences with entries in $\mathcal{I}$ and for
\[
\textbf{\emph{i}}= \big(i_1, i_2, \dots, i_k \big) \in \mathcal{I}^*
\]
write
\[
S_{\textbf{\emph{i}}} = S_{i_1} \circ S_{i_2} \circ \dots \circ S_{i_k}
\]
and $\alpha_1 (\textbf{\emph{i}}) \geq \alpha_2 (\textbf{\emph{i}})$ for the singular values of the linear part of the map $S_{\textbf{\emph{i}}}$.  Note that, for all $\textbf{\emph{i}} \in \mathcal{I}^*$, the singular values, $\alpha_1 (\textbf{\emph{i}})$ and $\alpha_2 (\textbf{\emph{i}})$, are just the lengths of the sides of the rectangle $S_{\textbf{\emph{i}}}\big([0,1]^2\big)$.  Finally, let
\[
\alpha_{\min} = \min \{\alpha_2(i) : i \in \mathcal{I} \}
\]
and
\[
\alpha_{\max} = \max \{\alpha_1(i) : i \in \mathcal{I} \}.
\]
Recall that the lower and upper box-counting dimensions of a bounded set $F \subset \mathbb{R}^d$ are defined by
\begin{equation*} \label{lowerbox}
\underline{\dim}_\text{B} F = \liminf_{\delta \to 0} \, \frac{\log N_\delta (F)}{-\log \delta}
\end{equation*}
and
\begin{equation*} \label{upperbox}
\overline{\dim}_\text{B} F = \limsup_{\delta \to 0} \,  \frac{\log N_\delta (F)}{-\log \delta}
\end{equation*}
respectively, where $N_\delta (F)$ is the smallest number of sets required for a $\delta$-cover of $F$, or, alternatively, the number of closed squares in a $\delta$-mesh which intersect $F$.  If $\underline{\dim}_\text{B} F = \overline{\dim}_\text{B} F$, then we call the common value the box-counting dimension of $F$ and denote it by $\dim_\text{B} F$.  Although we will always refer to $\dim_\text{B}$ as the box-counting dimension, or just box dimension, it is also commonly referred to by other names, for example, the entropy dimension or Minkowski dimension.  Our results also concern packing dimension, $\dim_\P$, however, we will not use its definition directly and so we omit it.  For the definitions of packing measure and dimension, as well as a discussion of various properties of box dimension and the interplay between packing and box dimension, the reader is referred to \cite{falconer}.

\section{Results} \label{results}

In this section we will state our main results.  The dimension formula, which relies on the knowledge of the dimensions of the projection of $F$ onto the horizontal and vertical axes, will be given in Section \ref{dimformula}.  In Section \ref{projections} we will discuss the problem of calculating the dimensions of the relevant projections.

\subsection{The dimension formula} \label{dimformula}

Let $\pi_1, \pi_2: \mathbb{R}^2 \to \mathbb{R}$ be defined by $\pi_1(x,y) = x$ and $\pi_2(x,y) = y$ respectively.  Also, let
\[
s_1 = \dim_\text{B} \pi_1(F)
\]
and
\[
s_2 = \dim_\text{B} \pi_2(F).
\]
It can be shown that both $\dim_\text{B} \pi_1(F)$ and $\dim_\text{B} \pi_2(F)$ exist using the `implicit theorems' found in \cite{implicit, mclaughlin}, or, alternatively, see Lemma \ref{proj2} in Section \ref{projections}.  For $\textbf{\emph{i}} \in \mathcal{I}^*$, let $b(\textbf{\emph{i}}) = \lvert \pi_1(S_{\textbf{\emph{i}}}[0,1]^2)\rvert$ and $h(\textbf{\emph{i}}) = \lvert \pi_2(S_{\textbf{\emph{i}}}[0,1]^2)\rvert$ denote the length of the base and height of the rectangle $S_{\textbf{\emph{i}}}[0,1]^2$ respectively and define $\pi_{\textbf{\emph{i}}}:\mathbb{R}^2 \to \mathbb{R}$ by
\[
\pi_{\textbf{\emph{i}}} = \left\{ \begin{array}{cc}
\pi_1 &  \text{if $\textbf{\emph{i}} \in \mathcal{I}_A$ and $b(\textbf{\emph{i}}) \geq h(\textbf{\emph{i}})$}\\ 
\pi_2 &   \text{if $\textbf{\emph{i}} \in \mathcal{I}_A$ and $b(\textbf{\emph{i}}) < h(\textbf{\emph{i}})$}\\
\pi_1 &   \text{if $\textbf{\emph{i}} \in \mathcal{I}_B$ and $b(\textbf{\emph{i}}) < h(\textbf{\emph{i}})$}\\ 
\pi_2 &   \text{if $\textbf{\emph{i}} \in \mathcal{I}_B$ and $b(\textbf{\emph{i}}) \geq h(\textbf{\emph{i}})$}
\end{array} \right. 
\]
Finally, let $s(\textbf{\emph{i}}) = \dim_\text{B} \pi_{\textbf{\emph{i}}}F$.  In fact, $s(\textbf{\emph{i}})$ is simply the box dimension of the projection of $S_\textbf{\emph{i}}(F)$ onto the longest side of the rectangle $S_\textbf{\emph{i}}\big([0,1]^2\big)$ and is always equal to either $s_1$ or $s_2$.
\\ \\
For $s\geq 0$ and $\textbf{\emph{i}} \in \mathcal{I}^*$, we define the \emph{modified singular value function}, $\psi^s$, of $S_{\textbf{\emph{i}}}$ by

\begin{equation} \label{modsing}
\psi^s\big(S_{\textbf{\emph{i}}}\big) = \alpha_1 (\textbf{\emph{i}})^{ s(\textbf{\emph{i}})} \, \,  \alpha_2 (\textbf{\emph{i}})^ {s-s(\textbf{\emph{i}})},
\end{equation}

and for $s\geq 0$ and $k \in \mathbb{N}$, we define a number $\Psi_k^s$ by
\[
\Psi_k^s= \sum_{\textbf{\emph{i}} \in \mathcal{I}^{k}} \psi^s(S_{\textbf{\emph{i}}}) .
\]

\begin{lma}[multiplicative properties] \hspace{1mm} \label{additive}
\\ \\
a) For $s \geq 0$ and $\textbf{{i}}, \textbf{{j}} \in \mathcal{I}^*$ we have

\begin{itemize}
\item[a1)]  If $s< s_1+s_2$, then $\psi^s(S_{\textbf{{i}}} \circ S_{\textbf{{j}}}) \leq  \psi^s(S_{\textbf{{i}}})  \, \psi^s(S_{\textbf{{j}}})$;

\item[a2)]  If $s= s_1+s_2$, then $\psi^s(S_{\textbf{{i}}} \circ S_{\textbf{{j}}}) =  \psi^s(S_{\textbf{{i}}})  \, \psi^s(S_{\textbf{{j}}})$;

\item[a3)]  If $s> s_1+s_2$, then $\psi^s(S_{\textbf{{i}}} \circ S_{\textbf{{j}}}) \geq  \psi^s(S_{\textbf{{i}}})  \, \psi^s(S_{\textbf{{j}}})$.

\end{itemize}

b) For $s \geq 0$ and $k,l \in \mathbb{N}$ we have

\begin{itemize}
\item[b1)]  If $s< s_1+s_2$, then $\Psi_{k+l}^s \leq \Psi_{k}^s \, \Psi_{l}^s$;

\item[b2)]  If $s= s_1+s_2$, then $\Psi_{k+l}^s = \Psi_{k}^s \, \Psi_{l}^s$;

\item[b3)]  If $s> s_1+s_2$, then $\Psi_{k+l}^s \geq \Psi_{k}^s \, \Psi_{l}^s$.

\end{itemize}

\end{lma}

We will prove Lemma \ref{additive} in Section \ref{add}.  It follows from Lemma \ref{additive} and standard properties of sub- and super-multiplicative sequences that that we may define a function $P:[0, \infty) \to [0, \infty)$ by:
\[
P(s) = \lim_{k \to \infty} (\Psi_k^s) ^{1/k} 
\]
where, in fact,
\[
\lim_{k \to \infty} (\Psi_k^s) ^{1/k} = \left\{ \begin{array}{cc}
\inf_{k \in \mathbb{N}} \,  (\Psi_k^s) ^{1/k}&  \text{if $s \in [0,s_1+s_2)$}\\ \\
\Psi_1^s &  \text{if $s  = s_1+s_2$}\\ \\
\sup_{k \in \mathbb{N}} \,  (\Psi_k^s) ^{1/k} &  \text{if $s \in (s_1+s_2, \infty)$}
\end{array} \right. 
\]
We think of $P$ as being akin to a Bowen-like pressure function.  Although $P$ is not a `pressure' function in the usual sense of the word, it is the exponential of the subadditive function
\[
P^*(s)  = \lim_{k \to \infty} \tfrac{1}{k} \log \Psi_k^s 
\]
which one might call the \emph{topological pressure} of the system.

\begin{lma}[Properties of $P$] \label{P} \hspace{1mm}
\begin{itemize}
\item[(1)]  For all $s,t \geq 0$ we have
\[
\alpha_{\min}^s P(t) \, \leq\,  P(s+t)\,  \leq \, \alpha_{\max}^s P(t)
\]
and furthermore, setting $t=0$, we have, for all $s \geq 0$ 
\[
0\, <\, \alpha_{\min}^s P(0) \, \leq\,  P(s)\,  \leq \, \alpha_{\max}^s P(0) \, < \, \infty,
\]
where $P(0) \in [\lvert \mathcal{I} \rvert,\infty)$ is a constant;
\item[(2)] $P$ is continuous on $[0, \infty)$;
\item[(3)] $P$ is strictly decreasing on $[0,\infty)$;
\item[(4)] There is a unique value $s\geq 0$ for which $P(s)=1$.
\end{itemize}
\end{lma}

We will prove Lemma \ref{P} in Section \ref{prelim}.  We can now state our main result concerning the packing and box-counting dimensions for box-like self-affine sets.

\begin{thm} \label{main}
Let $F$ be a box-like self-affine set.  Then $\dim_\text{\emph{P}} F = \overline{\dim}_\text{\emph{B}} F \leq s$ where $s \geq 0$ is the unique solution of $P(s) = 1$.  Furthermore, if the ROSC is satisfied, then $\dim_\text{\emph{P}} F = \dim_\text{\emph{B}} F= s$.
\end{thm}

We will prove Theorem \ref{main} in Section \ref{mainproof}.  We will now give two corollaries of Theorem \ref{main} which show that the dimension formula can be simplified in certain situations.  The first of which deals with the case where $s_1 = s_2$.  This will occur, for example, if $F$ is of non-separated type (see Lemma \ref{proj2}).

\begin{cor}
Let $F$ be a box-like self-affine set which satisfies the ROSC and is such that $s_1 = s_2 = :t$.  Then $\dim_\text{\emph{P}} F = \dim_\text{\emph{B}} F= s$, where $s$ satisfies
\[
\lim_{k \to \infty} \bigg(\sum_{\textbf{{i}} \in \mathcal{I}^{k}}  \alpha_1 (\textbf{{i}})^{t} \, \,  \alpha_2 (\textbf{{i}})^ {s-t}\bigg)^{1/k} = 1.
\]
\end{cor}

The second corollary deals with the case where $s_1 = s_2 = 1$.  Some easily verified sufficient conditions for this to occur are given in Lemma \ref{full}.

\begin{cor} \label{coro2}
Let $F$ be a box-like self-affine set which satisfies the ROSC and is such that $s_1 = s_2 = 1$.  Then 
\[
\dim_\text{\emph{P}} F = \dim_\text{\emph{B}} F= d
\]
where $d$ is the affinity dimension (\ref{affinitydim}).
\end{cor}

To prove Corollary \ref{coro2} simply observe that, if $s_1 = s_2 = 1$, then our modified singular value function (\ref{modsing}) coincides with the singular value function (\ref{singvaluefunc}) in the range $s \in [1,2]$.  Furthermore, it is clear that the dimension lies in this range and therefore the unique value of $s$ satisfying $P(s)=1$ is the affinity dimension.  The converse of Corollary \ref{coro2} is not true.  In particular, it is not true that if both $s_1$ and $s_2$ are strictly less than 1, then the packing dimension is strictly less than the affinity dimension (for example, some self-similar sets).  However, it is possible to give simple sufficient conditions for the packing dimension to drop from the affinity dimension.  For example, if both $s_1$ and $s_2$ are strictly less than $\min \{1, d\}$, where $d$ is the affinity dimension, and there exists a constant $\eta \in (0,1)$ such that for all $k \in \mathbb{N}$ and all $\textbf{\emph{i}} \in \mathcal{I}^k$, $\alpha_2(\textbf{\emph{i}}) \leq \eta^k \alpha_1(\textbf{\emph{i}})$, then the packing dimension of the attractor is strictly less than the affinity dimension.  To see this let $\epsilon = \min \{1, d\} - \max\{s_1,s_2\}>0$ and $d$ be the affinity dimension and note that
\[
P(d) = \inf_k \bigg(\sum_{\textbf{\emph{i}} \in \mathcal{I}^{k}} \psi^d(S_{\textbf{\emph{i}}}) \bigg)^{1/k} \leq \,\, \inf_k \Bigg(\sum_{\textbf{\emph{i}} \in \mathcal{I}^{k}} \phi^d(S_{\textbf{\emph{i}}}) \, \bigg(\frac{\alpha_2(\textbf{\emph{i}})}{\alpha_1(\textbf{\emph{i}})} \bigg)^\epsilon \Bigg)^{1/k} \leq \, \eta^\epsilon \,  < \, 1
\]
from which it follows that $\dim_\P F < d$.
\\ \\
Since $\overline{\dim}_\text{B} F \leq s_1+s_2$, it is clear that the solution of $P(s)=1$ always lies in the range $[0, s_1+s_2]$.  Even in the case where $s_1$ and $s_2$ can be computed it still may be very difficult to compute the solution of $P(s)=1$ explicitly.  However, since the solution lies in the submultiplicative region, it can be numerically estimated from above by considering the sequence $\{s_k\}_{k \in \mathbb{N}}$ where each $s_k$ is defined by $\Psi_k^{s_k} = 1$ and is an upper bound for the dimension.
\\ \\
We will now present one final corollary of Theorem \ref{main} which shows that for a certain class of box-like self-affine sets of separated type the dimension may be calculated explicitly due to the modified singular value function being multiplicative for all $s$ rather than sub- or supermultiplicative.

\begin{cor} \label{coro3}
Let $F$ be a box-like self-affine set of separated type which satisfies the ROSC.  Furthermore, assume that each map, $S_i$, in the IFS has singular values $\alpha_1(i) \geq \alpha_2(i)$ where the larger singular value, $\alpha_1(i)$, corresponds to contracting in the horizontal direction.  Then
\[
\dim_\text{\emph{P}} F = \dim_\text{\emph{B}} F= s
\]
where $s$ is the unique solution of
\[
\sum_{i \in \mathcal{I}} \alpha_1(i)^{s_1}\,  \alpha_2(i)^{s-s_1} = 1.
\]
Furthermore, if $s_1=1$, then $s$ is the affinity dimension.
\end{cor}

\begin{proof}
It may be gleaned from the proof of Lemma \ref{additive} (a1), case (i), that, in the situation described above, the modified singular value function is multiplicative.  It follows that the unique solution of $P(s)=1$ satisfies $\Psi_1^s = 1$.
\end{proof}

Corollary \ref{coro3} is similar to Corollary 1 in \cite{fengaffine} but our result covers a much larger class of sets since we allow the maps in the IFS to have non-trivial reflectional and rotational components (whilst ensuring that $F$ is of separated type).  Although in a different context, a problem related to Corollary \ref{coro3} was studied in \cite{Hu}.  There the author proved a version of Bowen's formula for a class of non-conformal $C^2$ expanding maps for which the expansion is stronger in one particular direction.
\\ \\
The idea to study box-like self-affine sets came from \cite{falcoc}.  There the authors consider self-similar sets and, in particular, how varying the rotational or reflectional component of the mappings affects the symmetry of the attractor.  Their approach relies on various group theoretic techniques.  One thing to note is that, given the OSC, changing the rotational or reflectional component of the mappings in an IFS of similarities does not change the dimension.  As we have shown (and unsurprisingly) the situation is more complicated in the self-affine case, see the examples below.  It would be interesting to conduct an analysis similar to that found in \cite{falcoc} in the self-affine case with the added complication that one could consider changes in dimension as well as changes in the symmetry of the (self-affine) attractor.

\subsection{Dimensions of projections} \label{projections}

The dimension formula given in Section \ref{dimformula} depends on knowledge of $s_1$ and $s_2$, i.e., the dimensions of the projections of $F$ onto the horizontal and vertical axes, respectively.  \emph{A priori}, $s_1$ and $s_2$ are difficult to calculate explicitly, or even to obtain good estimates for.  In this section we examine this problem and show that it is possible to compute $s_1$ and $s_2$ explicitly in a number of cases.

\begin{lma} \label{proj2}
If $F$ is of separated type, then $\pi_1(F)$ and $\pi_2(F)$ are self-similar sets.  If $F$ is of non-separated type, then $\pi_1(F)$ and $\pi_2(F)$ are a pair of graph-directed self-similar sets and, moreover, the associated adjacency matrix for the graph-directed system is irreducible.  In this second case, it follows that $s_1=s_2$.
\end{lma}

We will prove Lemma \ref{proj2} in Section \ref{projectionproofs2}.  It follows from Lemma \ref{proj2} that the box dimensions of the projections exist and so $s_1$ and $s_2$ are well-defined.  The problem with calculating the dimension of $\pi_1 (F)$ and $\pi_2 (F)$ is that the IFSs eluded to in Lemma \ref{proj2} may not satisfy the open set condition (OSC), or graph-directed open set condition (GDOSC).  However, in certain cases we will be able to invoke the \emph{finite type} conditions introduced in \cite{finitetype1, finitetype2, finitetype} and generalised to the graph-directed situation in \cite{gdfinitetype}.  In this situation, despite the possible failure of the OSC or GDOSC we can view the projections as attractors of alternative IFSs or graph-directed IFSs where the necessary separation conditions \emph{are} satisfied.  We can then compute $s_1$ and $s_2$ using a standard formula, see, for example, \cite{techniques}.  An example of this will be given in Section \ref{finitetypeex}.
\\ \\
There is one further situation where, even if the previously mentioned finite type conditions are not satisfied, we can still compute $s_1$ and $s_2$.  In this case we will say that $F$ is of \emph{block type}.

\begin{lma}[block type] \label{full}
Let $H$ be any closed, path connected set which contains $F$ and is not contained in any vertical or horizontal line.  If
\begin{equation} \label{full1}
\pi_1 \Big( \bigcup_{i \in \mathcal{I}} S_i H\Big) = \pi_1(H)
\end{equation}
and
\begin{equation} \label{full2}
\pi_2 \Big( \bigcup_{i \in \mathcal{I}} S_i H\Big) = \pi_2(H),
\end{equation}
then $s_1=s_2=1$.
\end{lma}

\begin{proof}
This follows immediately since, by (\ref{full1}) and (\ref{full2}), $\pi_1(F)$ and $\pi_2(F)$ are intervals.
\end{proof}

\section{Examples}

In order to illustrate our results we will now present two examples and compute the packing and box dimensions.  We will also examine what effect the rotational and reflectional components have on the dimension.  In both cases it will be clear that the ROSC is satisfied, taking $R = (0,1)^2$.  All rotations are taken to be clockwise about the origin and all numerical estimations were calculated in Maple using the method outlined at the end of Section \ref{results}.

\subsection{Non-separated type} \label{finitetypeex}

In this section we consider an example of a box-like self-affine set of non-separated type.  Let $F$ be the attractor of the IFS consisting of the maps which take $[0,1]^2$ to the 3 shaded rectangles on the left hand part of Figure 3, where the linear parts have been composed with: rotation by 270 degrees (top right); rotation by 90 degrees (bottom right); and reflection in the vertical axis (left).

\begin{figure}[H] 
	\centering
	\includegraphics[width=140mm]{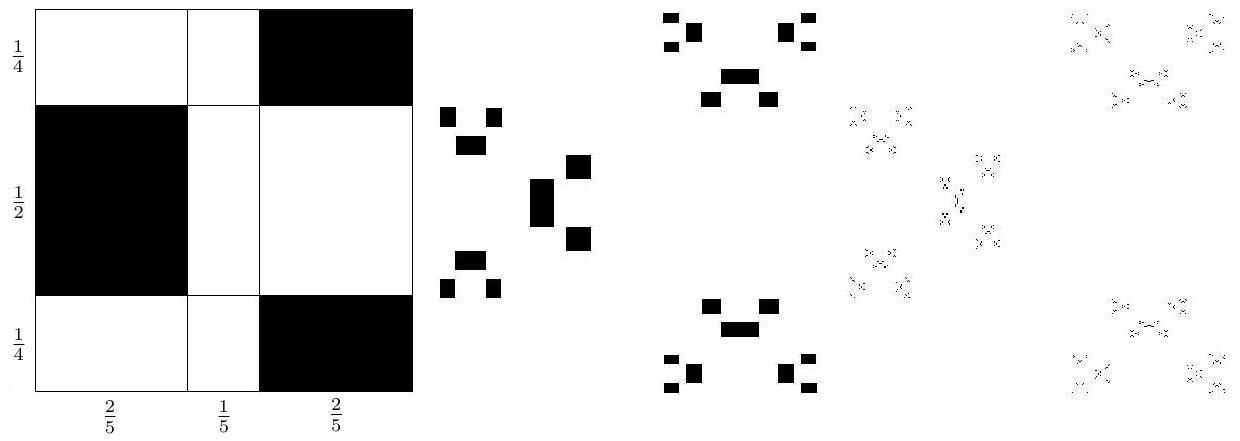}
\caption{Levels 1, 3 and 7 in the construction of $F$.}
\end{figure}
Here, $\pi_1(F)$ and $\pi_2(F)$ are a pair of graph-directed self-similar sets of finite type.  It is easy to see that in fact
\[
\pi_1(F) = \big(\tfrac{2}{5} - \tfrac{2}{5}\pi_1(F)\big) \cup \big(\tfrac{2}{5}\pi_2(F)+\tfrac{3}{5}\big)
\]
and
\[
\pi_2(F) = \big(\tfrac{1}{4}\pi_1(F)\big) \cup \big(\tfrac{1}{2}\pi_2(F)+\tfrac{1}{4}\big) \cup \big(1-\tfrac{1}{4}\pi_1(F)\big).
\]
with the GDOSC satisfied for this system.  The associated adjacency matrix is
\[
A^{(t)}=\left( \begin{array}{cc}
(\tfrac{2}{5})^t & (\tfrac{2}{5})^t \\ \\
2 \, (\tfrac{1}{4})^t & (\tfrac{1}{2})^t 
 \end{array} \right) 
\]
and solving $\rho\big(A^{(t)}\big) = 1$, where $\rho\big(A^{(t)}\big)$ is the spectral radius of $A^{(t)}$, for $t$ yields $s_1=s_2=:t \approx 0.890959$, see \cite{techniques}.  Theorem \ref{main} now gives that $\dim_\P F = \dim_\text{B} F= s$ where $s \geq 0$ is the unique solution of
\[
\lim_{k \to \infty} \bigg(\sum_{\textbf{\emph{i}} \in \mathcal{I}^k} \alpha_1(\textbf{\emph{i}})^{t}\,  \alpha_2(\textbf{\emph{i}})^{s-t}  \bigg)^{1/k} = 1,
\]
which was estimated numerically to be about 1.09.  If we considered the same construction but with no rotations or reflections then we would have a self-affine set of the type considered by Bara\'nski.  In this case, results  in \cite{baranski} give us that the box dimension is approximately 1.11349, which is certainly larger than the dimension we obtained for our construction.

\subsection{Block type}

In this section we consider an example of a box-like self-affine set of block type.  Let $F$ be the attractor of the IFS consisting of the maps $S_1$, $S_2$ and $S_3$ defined by
\[
S_1 = \left( \begin{array}{cc}
\tfrac{1}{2} & 0\\ 
0& \tfrac{3}{10}
 \end{array} \right) \circ R_1+ (0,1),
\]
\[
S_2 = \left( \begin{array}{cc}
\tfrac{1}{2} & 0\\
0& \tfrac{1}{5}
 \end{array} \right) \circ R_2+ (\tfrac{1}{4},\tfrac{7}{10})
\]
and
\[
S_3 = \left( \begin{array}{cc}
\tfrac{1}{4} & 0\\
0& \tfrac{3}{5}
 \end{array} \right) \circ R_3+ (1,0),
\]
where $R_1$ is reflection in the horizontal axis, $R_2$ is rotation by 90 degrees and $R_3$ is reflection in the vertical axis.
\begin{figure}[H]
	\centering
	\includegraphics[width=150mm]{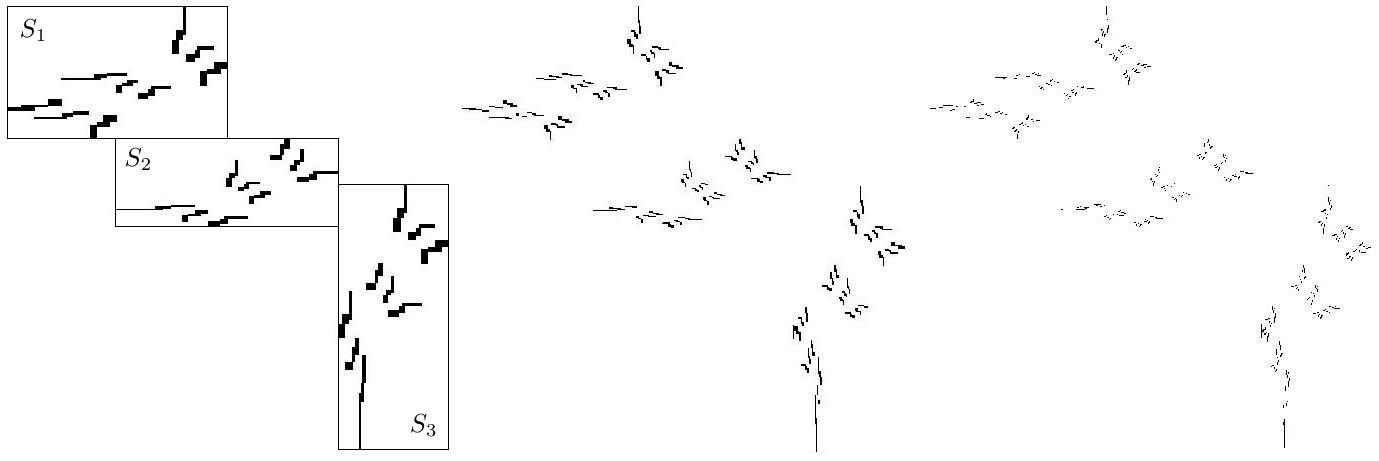}
\caption{Levels 4, 5 and 6 in the construction of $F$. The boxes in the first image on the left indicate the mappings.}
\end{figure}

It is clear that $F$ is of block type, taking $H = [0,1]^2$ in Lemma \ref{full}, and so $s_1=s_2=1$.  Theorem \ref{main} now gives that $\dim_\P F = \dim_\text{B} F= s$ where $s \geq 0$ is the unique solution of
\[
\lim_{k \to \infty} \bigg(\sum_{\textbf{\emph{i}} \in \mathcal{I}^k} \alpha_1(\textbf{\emph{i}})^1\,  \alpha_2(\textbf{\emph{i}})^{s-1}  \bigg)^{1/k} = 1.
\]
This was estimated numerically to be around 1.15, which, by Corollary \ref{coro2}, coincides with the affinity dimension.  Again, let us consider the same construction but with no rotational or reflectional components in the mappings.  In this case we have a self-affine set of the type considered by Feng and Wang and results in \cite{fengaffine} give that the box dimension is approximately 1.18405, which is again larger than for our construction.

\section{Proofs of preliminary lemmas} \label{prelim}

\subsection{Proof of Lemma \ref{additive}} \label{add}

We will first prove part (a) by a case by case analysis.  Part (b) will then follow easily.
\\ \\
\emph{Proof of (a).}
\\ \\
a1) Let $s \in [0, s_1+s_2)$ and let $\textbf{\emph{i}},\textbf{\emph{j}} \in \mathcal{I}^*$.  Firstly, assume that $F$ is of non-separated type.  It follows that $s_1 = s_2 = :t$.  We have
\begin{eqnarray*}
\psi^s(S_\textbf{\emph{i}} \circ S_\textbf{\emph{j}}) &=& \alpha_1( \textbf{\emph{i}} \textbf{\emph{j}})^{t}\alpha_2( \textbf{\emph{i}} \textbf{\emph{j}})^{s-t}\\ \\
&=& \Big(\alpha_1( \textbf{\emph{i}} \textbf{\emph{j}}) \, \alpha_2( \textbf{\emph{i}} \textbf{\emph{j}}) \Big)^{s-t} \, \alpha_1( \textbf{\emph{i}} \textbf{\emph{j}})^{2 t-s}\\ \\
&=& \Big(\alpha_1( \textbf{\emph{i}}) \, \alpha_2( \textbf{\emph{i}} ) \alpha_1(  \textbf{\emph{j}}) \, \alpha_2(  \textbf{\emph{j}}) \Big)^{s-t} \, \alpha_1( \textbf{\emph{i}} \textbf{\emph{j}})^{2 t-s} \\ \\
&\leq& \Big(\alpha_1( \textbf{\emph{i}}) \, \alpha_2( \textbf{\emph{i}} )  \Big)^{s-t}\, \Big( \alpha_1(  \textbf{\emph{j}})\alpha_2(  \textbf{\emph{j}}) \Big)^{s-t} \, \Big(\alpha_1( \textbf{\emph{i}}) \, \alpha_1( \textbf{\emph{j}})\Big)^{2 t-s} \qquad \text{since $2t-s>0$}\\ \\
&=& \psi^s(S_\textbf{\emph{i}}) \, \psi^s( S_\textbf{\emph{j}})
\end{eqnarray*}

proving (a1) in the non-separated case.  Secondly, assume that $F$ is of separated type and assume, in addition, that $b(\textbf{\emph{i}}) \geq h(\textbf{\emph{i}})$, recalling that $b(\textbf{\emph{i}})$ and $h(\textbf{\emph{i}})$ are the lengths of the base and height of the rectangle $S_{\textbf{\emph{i}}}[0,1]^2$ respectively.  The case where $b(\textbf{\emph{i}}) < h(\textbf{\emph{i}})$ is analogous.  We now have the following three cases:

\begin{itemize}
\item[(i)] $b(\textbf{\emph{j}}) \geq h(\textbf{\emph{j}})$ and $b(\textbf{\emph{ij}}) \geq h(\textbf{\emph{ij}})$;
\item[(ii)] $b(\textbf{\emph{j}}) < h(\textbf{\emph{j}})$ and $b(\textbf{\emph{ij}}) \geq h(\textbf{\emph{ij}})$;
\item[(iii)] $b(\textbf{\emph{j}}) < h(\textbf{\emph{j}})$ and $b(\textbf{\emph{ij}}) < h(\textbf{\emph{ij}})$.
\end{itemize}
The key property that we will utilise here is that, since $F$ is of separated type, $b(\textbf{\emph{ij}}) = b(\textbf{\emph{i}})\,b(\textbf{\emph{j}})$ and $h(\textbf{\emph{ij}}) = h(\textbf{\emph{i}})\,h(\textbf{\emph{j}})$.  Note that this precludes the case: $b(\textbf{\emph{j}}) \geq h(\textbf{\emph{j}})$ and $b(\textbf{\emph{ij}}) < h(\textbf{\emph{ij}})$.  To complete the proof of (a1) we will show that, in each of the above cases (i-iii), we have
\[
\frac{\psi^s(S_{\textbf{\emph{i}}} \circ S_{\textbf{\emph{j}}})}{ \psi^s(S_\textbf{\emph{i}})  \, \psi^s(S_{\textbf{\emph{j}}})} \leq 1.
\]
(i)  We have
\[
\frac{\psi^s(S_{\textbf{\emph{i}}} \circ S_{\textbf{\emph{j}}})}{ \psi^s(S_\textbf{\emph{i}})  \, \psi^s(S_{\textbf{\emph{j}}})} = \frac{ b(\textbf{\emph{ij}})^{s_1} h(\textbf{\emph{ij}})^{s-s_1}}{b(\textbf{\emph{i}})^{s_1} h(\textbf{\emph{i}})^{s-s_1} b(\textbf{\emph{j}})^{s_1} h(\textbf{\emph{j}})^{s-s_1}}
= 1.
\]

(ii)  Similarly 
\[
\frac{\psi^s(S_{\textbf{\emph{i}}} \circ S_{\textbf{\emph{j}}})}{ \psi^s(S_\textbf{\emph{i}})  \, \psi^s(S_{\textbf{\emph{j}}})} = \frac{ b(\textbf{\emph{ij}})^{s_1} h(\textbf{\emph{ij}})^{s-s_1}}{b(\textbf{\emph{i}})^{s_1} h(\textbf{\emph{i}})^{s-s_1} h(\textbf{\emph{j}})^{s_2} b(\textbf{\emph{j}})^{s-s_2}}
= \bigg(\frac{ b(\textbf{\emph{j}})}{h(\textbf{\emph{j}})} \bigg)^{s_1+s_2-s} \leq 1.
\]

(iii)  Finally 
\[
\frac{\psi^s(S_{\textbf{\emph{i}}} \circ S_{\textbf{\emph{j}}})}{ \psi^s(S_\textbf{\emph{i}})  \, \psi^s(S_{\textbf{\emph{j}}})} = \frac{ h(\textbf{\emph{ij}})^{s_2} b(\textbf{\emph{ij}})^{s-s_2}}{b(\textbf{\emph{i}})^{s_1} h(\textbf{\emph{i}})^{s-s_1} h(\textbf{\emph{j}})^{s_2} b(\textbf{\emph{j}})^{s-s_2}}
= \bigg(\frac{ h(\textbf{\emph{i}})}{b(\textbf{\emph{i}})} \bigg)^{s_1+s_2-s} \leq 1.
\]
The proofs of (a2) and (a3) are similar and, therefore, omitted.
\\ \\
\emph{Proof of (b).}
\\ \\
This follows easily by noting that, for all $k, l \in \mathbb{N}$, we have
\[
\Psi_{k+l}^s = \sum_{\textbf{\emph{i}} \in \mathcal{I}^{k+l}} \psi^s(S_{\textbf{\emph{i}}}) = \sum_{\textbf{\emph{i}} \in \mathcal{I}^{k}}  \sum_{\textbf{\emph{j}} \in \mathcal{I}^{l}}\psi^s(S_{\textbf{\emph{i}}} \circ S_{\textbf{\emph{j}}}) 
\]
and
\[
\Psi_{k}^s \, \Psi_{l}^s = \Bigg(\sum_{\textbf{\emph{i}} \in \mathcal{I}^{k}} \psi^s(S_{\textbf{\emph{i}}})\Bigg) \Bigg( \sum_{\textbf{\emph{i}} \in \mathcal{I}^{l}} \psi^s(S_{\textbf{\emph{j}}}) \Bigg) = \sum_{\textbf{\emph{i}} \in \mathcal{I}^{k}}  \sum_{\textbf{\emph{j}} \in \mathcal{I}^{l}}\psi^s(S_{\textbf{\emph{i}}})  \, \psi^s(S_{\textbf{\emph{j}}})
\]
and applying part (a). \hfill \qed

\subsection{Proof of Lemma \ref{P}}

(1)  Let $s, t \in [0, \infty)$.  We have
\[
P(s+t) = \lim_{k \to \infty} \Bigg(\sum_{\textbf{\emph{i}} \in \mathcal{I}^{k}}  \alpha_1 (\textbf{\emph{i}})^{ s(\textbf{\emph{i}})} \, \,  \alpha_2 (\textbf{\emph{i}})^ {s+t-s(\textbf{\emph{i}})} \Bigg)^{1/k}
\leq \lim_{k \to \infty} \Bigg(\alpha_{\max}^{ks}\sum_{\textbf{\emph{i}} \in \mathcal{I}^{k}}  \alpha_1 (\textbf{\emph{i}})^{ s(\textbf{\emph{i}})} \, \,  \alpha_2 (\textbf{\emph{i}})^ {t-s(\textbf{\emph{i}})} \Bigg)^{1/k} = \alpha_{\max}^{s} P(t).
\]
The proof of the left hand inequality is similar.  Furthermore, note that
\[
\infty \, \, > \, \, \inf_{k \geq 0} (\Psi^0_k)^{\frac{1}{k}} \, \, = \, \, P(0) \, \, = \, \, \lim_{k \to \infty} \Bigg(\sum_{\textbf{\emph{i}} \in \mathcal{I}^{k}}  \alpha_1 (\textbf{\emph{i}})^{ s(\textbf{\emph{i}})} \, \,  \alpha_2 (\textbf{\emph{i}})^ {-s(\textbf{\emph{i}})} \Bigg)^{1/k} \geq \, \,   \lim_{k \to \infty} \Bigg(\sum_{\textbf{\emph{i}} \in \mathcal{I}^{k}} 1 \Bigg)^{1/k} = \,\,\, \, \lvert \mathcal{I} \rvert
\]
and together with setting $t=0$ above gives the second chain of inequalities.
\\ \\
(2)  The continuity of $P$ follows immediately from (1).
\\ \\
(3)  Let $t, \varepsilon \geq 0$.  Since $P(t+\varepsilon), P(t) \in (0, \infty)$, by (1) we have
\[
\frac{P(t+\varepsilon)}{P(t)} \leq \alpha_{\max}^\varepsilon < 1
\]
and so $P$ is strictly decreasing on $[0,\infty)$.
\\ \\
(4)  It follows from (1) that $P(0) \geq \lvert \mathcal{I} \rvert>1$ and that $P(s)<1$ for sufficiently large $s$.  These facts, combined with parts (2) and (3), imply that there is a unique value of $s$ for which $P(s)=1$. \hfill \qed

\subsection{Proof of Lemma \ref{proj2}} \label{projectionproofs2}

Let $\mathbb{I}_1,\mathbb{I}_2 : [0,1] \to [0,1]^2$ be defined by
\[
\mathbb{I}_1(x) = (x,0)
\]
and
\[
\mathbb{I}_2(y) = (0,y).
\]
Also, for $i \in \mathcal{I}$ and $a,b \in \{1,2\}$, we define a contracting similarity mapping $\tilde S_{i}^{a,b}: [0,1] \to [0,1]$ by
\[
\tilde S_{i}^{a,b} = \pi_a \circ S_{i} \circ \mathbb{I}_b.
\]
For certain choices of $a$, $b$ and $i$ the image $\tilde S_{i}^{a,b}([0,1])$ is a singleton.  We will not be interested in these maps.  Also, let $X = \pi_1 (F)$ and $Y = \pi_2 (F)$.  It is clear that

\begin{equation} \label{gdifs1}
X = \bigg(\bigcup_{i \in \mathcal{I}_A} \tilde S_{i}^{1,1} (X) \bigg) \cup \bigg(\bigcup_{i \in \mathcal{I}_B} \tilde S_{i}^{1,2} (Y) \bigg)
\end{equation}

and

\begin{equation} \label{gdifs2}
Y = \bigg(\bigcup_{i \in \mathcal{I}_A} \tilde S_{i}^{2,2} (Y) \bigg) \cup \bigg(\bigcup_{i \in \mathcal{I}_B} \tilde S_{i}^{2,1} (X) \bigg).
\end{equation}

It follows that if $\mathcal{I}_B = \emptyset$, then $X$ and $Y$ are the self-similar attractors of the IFSs $\{ \tilde S_{i}^{1,1}\}_{i \in \mathcal{I}}$ and $\{ \tilde S_{i}^{2,2}\}_{i \in \mathcal{I}}$ respectively and if $\mathcal{I}_B \neq \emptyset$, then $X$ and $Y$ are a pair of graph-directed self-similar sets with an irreducible associated adjacency matrix defined by (\ref{gdifs1}--\ref{gdifs2}).  This proves Lemma \ref{proj2}. \hfill \qed

\section{Proof of Theorem \ref{main}} \label{mainproof}

We will now prove our main result, that the packing and box-counting dimensions of $F$ are equal to the unique $s$ which satisfies $P(s)=1$.  We will prove this in the box dimension case and it is well-known that, since $F$ is compact and every open ball centered in $F$ contains a bi-Lipshitz image of $F$, $\dim_\P F = \overline{\dim}_\text{B} F$, see \cite{falconer}.
\\ \\
Let $s \geq 0$ be the unique solution of $P(s)=1$.  For
\[
\textbf{\emph{i}} = (i_1,i_2, \dots, i_{k-1}, i_k)  \in \mathcal{I}^*
\]
let
\[
 \overline{\textbf{\emph{i}}} = (i_1,i_2, \dots, i_{k-1}) \in \mathcal{I}^* \cup \{ \omega\},
\]
where $\omega$ is the empty word.  Note that the map $S_\omega$ is taken to be the identity map, which has singular values both equal to 1.  For $\delta \in (0,1]$ we define the $\delta$-\emph{stopping}, $\mathcal{I}_\delta$, as follows:
\[
\mathcal{I}_\delta = \big\{\textbf{\emph{i}} \in \mathcal{I}^* : \alpha_2(\textbf{\emph{i}}) < \delta \leq \alpha_2( \overline{\textbf{\emph{i}}}) \big\}.
\]
Note that for $\textbf{\emph{i}} \in \mathcal{I}_\delta$ we have
\begin{equation} \label{stoppingest}
\alpha_{\min} \, \delta \leq \alpha_2(\textbf{\emph{i}}) < \delta.
\end{equation}

\begin{lma} \label{deltaconv}
Let $t \geq 0$.

\begin{itemize}

\item[(1)]  If $t>s$, then there exists a constant $K(t)<\infty$ such that
\[
\sum_{\textbf{{i}} \in \mathcal{I}_\delta}\psi^t(S_\textbf{{i}}) \leq K(t)
\]
for all $\delta \in (0,1]$.

\item[(2)] If $t< s$,  then there exists a constant $L(t)>0$ such that
\[
 \sum_{\textbf{{i}} \in \mathcal{I}_\delta}\psi^t(S_\textbf{{i}}) \geq L(t)
\]
for all $\delta \in (0,1]$.

\end{itemize}
\end{lma}

\begin{proof}
(1)  Let $t>s$ and $\delta \in (0,1]$.  We have
\[
\sum_{\textbf{\emph{i}} \in \mathcal{I}_\delta} \psi^t(S_\textbf{\emph{i}}) \leq \sum_{\textbf{\emph{i}}\in \mathcal{I}^*}\psi^t(S_\textbf{\emph{i}})= \sum_{k=1}^\infty \sum_{\textbf{\emph{i}} \in \mathcal{I}^k}\psi^t(S_\textbf{\emph{i}}) = \sum_{k=1}^\infty \Psi_k^t < \infty
\]
since $\lim_{k \to \infty} (\Psi_k^t)^{1/k} = P(t) < 1$.  The result follows, setting $K(t) = \sum_{k=1}^\infty \Psi_k^t$.
\\ \\
(2)  Let $t< s$.  Consider two cases according to whether $t$ is in the submultiplicative region $[0, s_1+s_2]$, or supermultiplicative region $(s_1+s_2, \infty)$.  We will be able to deduce retrospectively that $s \leq s_1+s_2$ and so the second case is, in fact, vacuous.  It would be possible to prove part (2) only in the submultiplicative case and then obtain the result that the dimension is given by $\min\{s, s_1+s_2\}$ but in order to conclude that the dimension is simply $s$, we include the proof in the supermultiplicative case.
\\ \\
(i) $0 \leq t \leq s_1+s_2$.  We remark that an argument similar to the following was used in \cite{affine}.
\\ \\
Let $\delta \in (0,1]$ and assume that
\begin{equation} \label{bounded}
\sum_{\textbf{\emph{i}} \in \mathcal{I}_\delta} \psi^t(S_\textbf{\emph{i}}) \leq 1.
\end{equation}
To obtain a contradiction we will show that this implies that $t \geq s$.  Let $k(\delta) = \max \{ \lvert \textbf{\emph{i}} \rvert : \textbf{\emph{i}} \in \mathcal{I}_\delta \}$, where $\lvert \textbf{\emph{i}} \rvert$ denotes the length of the string $\textbf{\emph{i}}$, and define $\mathcal{I}_{\delta, k}$ by
\[
\mathcal{I}_{\delta, k} = \big\{\textbf{\emph{i}}_1 \dots \textbf{\emph{i}}_m : \textbf{\emph{i}}_j  \in \mathcal{I}_\delta \text{ for all $j = 1, \dots, m$}, \quad \text{$\lvert \textbf{\emph{i}}_1 \dots \textbf{\emph{i}}_m\rvert \leq k$  but  $\lvert \textbf{\emph{i}}_1 \dots \textbf{\emph{i}}_{m} \textbf{\emph{i}}_{m+1}\rvert > k$  for some  $\textbf{\emph{i}}_{m+1} \in \mathcal{I}_\delta$}  \big\}.
\]
For all $\textbf{\emph{i}} \in \mathcal{I}^*$ we have, by the submultiplicativity of $\psi^t$,
\begin{eqnarray*}
 \sum_{\textbf{\emph{j}} \in \mathcal{I}_\delta} \psi^t(S_{\textbf{\emph{i}} \, \textbf{\emph{j}}}) &\leq&  \sum_{\textbf{\emph{j}} \in \mathcal{I}_\delta} \psi^t(S_\textbf{\emph{i}} ) \,  \psi^t(S_\textbf{\emph{j}}) \\ \\
&=& \psi^t(S_\textbf{\emph{i}}) \, \sum_{\textbf{\emph{j}} \in \mathcal{I}_\delta}   \psi^t(S_\textbf{\emph{j}})\\ \\
&<& \psi^t(S_\textbf{\emph{i}})
\end{eqnarray*}
by (\ref{bounded}).  It follows by repeated application of the above that, for all $k \in \mathbb{N}$,
\begin{equation} \label{bounded2old}
\sum_{\textbf{\emph{i}} \in \mathcal{I}_{\delta,k}} \psi^t(S_\textbf{\emph{i}}) \leq 1.
\end{equation}
Let $\textbf{\emph{i}} \in \mathcal{I}^{k}$ for some $k \in \mathbb{N}$.  It follows that $\textbf{\emph{i}} = \textbf{\emph{j}}_1 \, \textbf{\emph{j}}_2$ for some $\textbf{\emph{j}}_1 \in \mathcal{I}_{\delta,k}$ and some $\textbf{\emph{j}}_2 \in \mathcal{I}^* \cup \{\omega\}$ with $\lvert \textbf{\emph{j}}_2 \rvert \leq k(\delta)$ and by the submultiplicativity of $\psi^t$,
\[
\psi^t(S_\textbf{\emph{i}})\, = \,\psi^t(S_{\textbf{\emph{j}}_1 \, \textbf{\emph{j}}_2}) \, \leq \,  \psi^t(S_{\textbf{\emph{j}}_1}) \, \psi^t( S_{\textbf{\emph{j}}_2}) \, \leq \, c_{k(\delta)} \,  \psi^t(S_{\textbf{\emph{j}}_1} ),
\]
where $c_{k(\delta)} = \max\{ \psi^t( S_{\textbf{\emph{i}}}) : \lvert \textbf{\emph{i}} \rvert \leq k(\delta) \}< \infty$ is a constant which depends only on $\delta$.  Since there are at most $\lvert \mathcal{I}\rvert^{k(\delta)+1}$ elements $\textbf{\emph{j}}_2 \in \mathcal{I}^*\cup \{\omega\}$ with $\lvert \textbf{\emph{j}}_2 \rvert \leq k(\delta)$ we have
\[
\Psi^t_{k} \, \, = \, \, \sum_{\textbf{\emph{i}} \in \mathcal{I}^{k}} \psi^t(S_\textbf{\emph{i}})  \, \, \leq  \, \, \lvert \mathcal{I} \rvert^{k(\delta)+1} \,c_{k(\delta)} \,  \sum_{\textbf{\emph{i}} \in \mathcal{I}_{\delta, k}} \psi^t(S_\textbf{\emph{i}})
 \, \, \leq  \, \, \lvert \mathcal{I} \rvert^{k(\delta)+1} \,c_{k(\delta)}
\]
by (\ref{bounded2old}).  Since this is true for all $k \in \mathbb{N}$ we have
\[
P(t) = \lim_{k \to \infty} \big(\Psi^t_{k}\big)^{1/k} \leq 1
\]
from which it follows that $t \geq s$.  So, if $t \leq s_1+s_2$, then we may set $L(t) = 1$.
\\ \\
(ii) $t > s_1+s_2$.
\\ \\
Since $t<s$ it follows that $\sum_{\textbf{\emph{i}} \in \mathcal{I}^k} \psi^t(S_\textbf{\emph{i}}) \to \infty$ as $k \to \infty$.  Therefore, we may fix a $k \in \mathbb{N}$ such that
\begin{equation} \label{bigger1}
\sum_{\textbf{\emph{i}} \in \mathcal{I}^k} \psi^t(S_\textbf{\emph{i}}) \geq 1.
\end{equation}
Fix $\delta \in (0,1]$ and define
\begin{eqnarray*}
\mathcal{I}_{k,\delta} = \big\{\textbf{\emph{i}}_1 \dots \textbf{\emph{i}}_m &:&  \textbf{\emph{i}}_j  \in \mathcal{I}^k \text{ for all $j = 1, \dots, m$, } \quad \\ \\
&\,&\text{ $\alpha_2(\textbf{\emph{i}}_1 \dots \textbf{\emph{i}}_m) \geq \delta$ but $\alpha_2(\textbf{\emph{i}}_1 \dots \textbf{\emph{i}}_m \textbf{\emph{i}}_{m+1}) < \delta$ for some $\textbf{\emph{i}}_{m+1} \in \mathcal{I}^k$} \big\}.
\end{eqnarray*}
For all $\textbf{\emph{i}} \in \mathcal{I}^*$ we have, by the supermultiplicativity of $\psi^t$,
\begin{eqnarray*}
 \sum_{\textbf{\emph{j}} \in \mathcal{I}^k} \psi^t(S_{\textbf{\emph{i}} \, \textbf{\emph{j}}}) &\geq&  \sum_{\textbf{\emph{j}} \in \mathcal{I}^k} \psi^t(S_\textbf{\emph{i}}) \,  \psi^t(S_\textbf{\emph{j}})\\ \\
&=& \psi^t(S_\textbf{\emph{i}} ) \, \sum_{\textbf{\emph{j}} \in \mathcal{I}^k}   \psi^t(S_\textbf{\emph{j}})\\ \\
&\geq& \psi^t(S_\textbf{\emph{i}} )
\end{eqnarray*}
by (\ref{bigger1}).  It follows by repeated application of the above that
\begin{equation} \label{bounded2}
\sum_{\textbf{\emph{i}} \in \mathcal{I}_{k,\delta}} \psi^t(S_\textbf{\emph{i}}) \geq 1.
\end{equation}
Let $\textbf{\emph{i}} \in \mathcal{I}_{\delta}$.  It follows that $\textbf{\emph{i}} = \textbf{\emph{j}}_1 \textbf{\emph{j}}_2$ for some $\textbf{\emph{j}}_1 \in \mathcal{I}_{k,\delta}$ and some $\textbf{\emph{j}}_2 \in \mathcal{I}^*$.  Since $\alpha_2(\textbf{\emph{i}}) \geq \delta \, \alpha_{\min}$ by (\ref{stoppingest}) and $\alpha_2(\textbf{\emph{j}}_1) \leq \delta \alpha_{\min}^{-k}$ we have
\begin{equation} \label{lengthest}
\alpha_2(\textbf{\emph{j}}_1) \leq \alpha_2(\textbf{\emph{i}}) \alpha_{\min}^{-(k+1)} \leq \alpha_2(\textbf{\emph{j}}_1) \alpha_{\max}^{\lvert \textbf{\emph{j}}_2\rvert} \alpha_{\min}^{-(k+1)}
\end{equation}
which yields $\lvert \textbf{\emph{j}}_2 \rvert \leq (k+1)\frac{\log \alpha_{\min}}{\log \alpha_{\max}}$.  Setting $c_k  =  \min \Big\{ \psi^t(S_\textbf{\emph{i}}) : \lvert \textbf{\emph{i}}\rvert \leq (k+1) \frac{\log \alpha_{\min}}{\log \alpha_{\max}} \Big\} >0$ it follows from (\ref{lengthest}), (\ref{bounded2}) and the supermultiplicativity of $\psi^t$ that
\[
\sum_{\textbf{\emph{i}} \in \mathcal{I}_{\delta}} \psi^t(S_\textbf{\emph{i}}) \geq c_k  \sum_{\textbf{\emph{i}} \in \mathcal{I}_{k,\delta}}  \psi^t(S_\textbf{\emph{i}}) \geq c_k.
\]
We have now proved part (2) setting $L(t) = \min \{1, \, c_k\} = c_k$.  Note that although $L(t)$ appears to depend on $k$, recall that we fixed $k$ at the beginning of the proof of (2)(ii) and the choice of $k$ depended only on $t$.
\end{proof}

We are now ready to prove Theorem \ref{main}.  It follows immediately from the definition of box dimension that for all $\varepsilon>0$ there exists $C_\varepsilon \geq 1$ such that for all $\delta>0$ we have
\begin{equation} \label{simplebox1}
\tfrac{1}{C_\varepsilon} \, \delta^{-s_1 + \varepsilon/2} \leq N_\delta(\pi_1 F) \leq C_\varepsilon \, \delta^{-s_1 - \varepsilon/2}
\end{equation}
and
\begin{equation} \label{simplebox2}
\tfrac{1}{C_\varepsilon} \, \delta^{-s_2 + \varepsilon/2} \leq N_\delta(\pi_2 F) \leq C_\varepsilon \, \delta^{-s_2 - \varepsilon/2}.
\end{equation}
\textbf{Upper bound (assuming no separation conditions)}
\\ \\
Let $\varepsilon >0$, $\delta>0$ and suppose that, for each $\textbf{\emph{i}}\in \mathcal{I}_\delta$,  $\{U_{\textbf{\emph{i}}, j}\}_{j=1}^{N_\delta (F_{\textbf{\emph{i}}})}$ is a $\delta$-cover of $F_{\textbf{\emph{i}}}$.  Since $F = \bigcup_{\textbf{\emph{i}} \in \mathcal{I}_\delta} F_{\textbf{\emph{i}}}$ it follows that
\[
\bigcup_{\textbf{\emph{i}} \in \mathcal{I}_\delta}\bigcup_{j=1}^{N_\delta (F_{\textbf{\emph{i}}})} \{U_{\textbf{\emph{i}}, j}\}
\]
is a $\delta$-cover for $F$.  Whence,
\begin{eqnarray*}
0 \quad \leq \quad  \delta^{s+\varepsilon} N_\delta (F) &\leq& \delta^{s+\varepsilon} \sum_{\textbf{\emph{i}} \in \mathcal{I}_\delta} N_\delta \big(F_{\textbf{\emph{i}}}\big)\\ \\
&=&  \delta^{s+\varepsilon} \sum_{\textbf{\emph{i}} \in \mathcal{I}_\delta} N_{\delta/\alpha_1(\textbf{\emph{i}})} \big(\pi_{\textbf{\emph{i}}}F\big) \qquad \qquad \text{since $\alpha_2(\textbf{\emph{i}}) < \delta$}\\ \\
&\leq&   \delta^{s+\varepsilon} \sum_{\textbf{\emph{i}} \in \mathcal{I}_\delta} C_\varepsilon \,  \bigg(\frac{\delta}{\alpha_1(\textbf{\emph{i}})} \bigg)^{-s(\textbf{\emph{i}})-\varepsilon/2} \qquad \qquad \text{by (\ref{simplebox1}--\ref{simplebox2})}\\ \\
&\leq&  C_\varepsilon \, \alpha_{\min}^{-s-\epsilon} \sum_{\textbf{\emph{i}} \in \mathcal{I}_\delta}\alpha_1(\textbf{\emph{i}})^{s(\textbf{\emph{i}})+\varepsilon/2}\alpha_2(\textbf{\emph{i}})^{s+\varepsilon-s(\textbf{\emph{i}})-\varepsilon/2}\qquad \qquad \text{by (\ref{stoppingest})}\\ \\
&\leq&   C_\varepsilon \,\alpha_{\min}^{-s-\epsilon} \sum_{\textbf{\emph{i}} \in \mathcal{I}_\delta} \psi^{s+\epsilon/2}(\textbf{\emph{i}})\\ \\
&\leq&  C_\varepsilon \,\alpha_{\min}^{-s-\varepsilon} \, K\big(s+\tfrac{\varepsilon}{2}\big)
\end{eqnarray*}
by Lemma \ref{deltaconv} (1).  It follows that $\overline{\dim}_\text{B} F \leq s+\varepsilon$ and, since $\varepsilon>0$ was arbitrary, we have the desired upper bound.
\\ \\
\textbf{Lower bound (assuming the ROSC)}
\\ \\
Let $\varepsilon \in (0, s)$, $\delta>0$ and $U$ be any closed square of sidelength $\delta$.  Also, let $R$ be the open rectangle used in the ROSC and let $r_-$ denote the length of the shortest side of $R$.  Finally, let
\[
M = \min\big\{n \in \mathbb{N} : n \geq (\alpha_{\min}r_-)^{-1}+1\big\}.
\]
Since $\{S_\textbf{\emph{i}}(R)\}_{\textbf{\emph{i}}\in \mathcal{I}_\delta}$ is a collection of pairwise disjoint open rectangles each with shortest side having length at least $\alpha_{\min} \delta r_-$, it is clear that $U$ can intersect no more that $M^2$ of the sets $\{ F_{\textbf{\emph{i}}}\}_{\textbf{\emph{i}}\in \mathcal{I}_\delta}$.  It follows that, using the $\delta$-mesh definition of $N_\delta$, we have
\[
\sum_{\textbf{\emph{i}}\in \mathcal{I}_\delta} N_\delta \big(F_{\textbf{\emph{i}}}\big)  \leq M^2 \, N_\delta (F).
\]
This yields
\begin{eqnarray*}
\delta^{s-\varepsilon} N_\delta (F) &\geq&   \delta^{s-\varepsilon} \, \tfrac{1}{M^2}\sum_{\textbf{\emph{i}}\in \mathcal{I}_\delta} N_\delta \big(F_{\textbf{\emph{i}}}\big)\\ \\
&=&\delta^{s-\varepsilon} \, \tfrac{1}{M^2}   \sum_{\textbf{\emph{i}} \in \mathcal{I}_\delta} N_{\delta/\alpha_1(\textbf{\emph{i}})}\big(\pi_{\textbf{\emph{i}}}F\big) \qquad \qquad \text{since $\alpha_2(\textbf{\emph{i}}) < \delta$}\\ \\
&\geq& \delta^{s-\varepsilon} \, \tfrac{1}{M^2} \sum_{\textbf{\emph{i}} \in \mathcal{I}_\delta} \tfrac{1}{C_\varepsilon}\bigg(\frac{\delta}{\alpha_1(\textbf{\emph{i}})} \bigg)^{-s(\textbf{\emph{i}})+\varepsilon/2} \qquad \qquad \text{by (\ref{simplebox1}--\ref{simplebox2})} \\ \\
&\geq& \tfrac{1}{M^2 C_\varepsilon} \, \sum_{\textbf{\emph{i}} \in \mathcal{I}_\delta}  \alpha_2(\textbf{\emph{i}})^{s-\varepsilon} \, \alpha_{\min} \, \alpha_2(\textbf{\emph{i}})^{-s(\textbf{\emph{i}})+\varepsilon/2} \, \alpha_1(\textbf{\emph{i}})^{s(\textbf{\emph{i}})-\varepsilon/2}  \qquad \qquad \text{by (\ref{stoppingest})}\\ \\
&=& \tfrac{1}{M^2 C_\varepsilon} \, \alpha_{\min} \, \sum_{\textbf{\emph{i}} \in \mathcal{I}_\delta} \alpha_1(\textbf{\emph{i}})^{s(\textbf{\emph{i}})-\varepsilon/2} \,  \alpha_2(\textbf{\emph{i}})^{s-\varepsilon/2-s(\textbf{\emph{i}})}  \\ \\
&\geq&\tfrac{1}{M^2 C_\varepsilon} \, \alpha_{\min} \, \sum_{\textbf{\emph{i}} \in \mathcal{I}_\delta} \psi^{s-\varepsilon/2}(\textbf{\emph{i}})  \\ \\
&\geq& \tfrac{1}{M^2} \tfrac{1}{C_\varepsilon}\, \alpha_{\min} \, L\big(s-\tfrac{\varepsilon}{2}\big)
\end{eqnarray*}
by Lemma \ref{deltaconv} (2).  It follows that $\underline{\dim}_\text{B} F \geq s-\varepsilon$ and, since $\varepsilon \in (0, s)$ was arbitrary, we have the desired lower bound. \hfill $\qed$

\begin{centering}

\textbf{Acknowledgements}

\end{centering}

The author thanks Kenneth Falconer for many helpful discussions during the writing of this paper.  He was supported by an EPSRC Doctoral Training Grant.

\end{document}